\documentclass[12pt]{amsart}

\usepackage{amsmath}
\usepackage{amssymb}
\usepackage{hyperref}
\usepackage[all]{xy}
\usepackage{tikz}

\usepackage{soul}

\theoremstyle{plain}
\newtheorem{theorem}{Theorem}[section]
\newtheorem{proposition}[theorem]{Proposition}

\newtheorem{lemma}[theorem]{Lemma}

\newtheorem{definition}[theorem]{Definition}

\theoremstyle{remark}
\newtheorem{remark}[theorem]{Remark}
\newtheorem{example}[theorem]{Example}

\numberwithin{equation}{section}

\newcommand{\K}{\ensuremath{\Bbbk}}
\newcommand{\N}{\mathbb N}
\newcommand{\Z}{\mathbb Z}
\newcommand{\s}{\mathbb S}

\newcommand{\R}{\mathcal R}

\newcommand{\MC}{\mathcal MC}

\newcommand{\rad}{\operatorname{rad}}
\newcommand{\Hom}{\operatorname{Hom}}

\newcommand{\Ext}{\operatorname{Ext}}
\newcommand{\sgn}{\operatorname{sgn}}
\renewcommand{\Im}{\operatorname{Im}}

\newcommand{\Id}{\textsl{Id}}

\newcommand{\HH}{\ensuremath{\mathsf{HH}}}

\begin{document}

\title[$L_\infty$-structure on Barzdell's complex]{$L_\infty$-structure on Barzdell's complex for monomial algebras} 

\author[M. J. Redondo and F. Rossi Bertone]{Mar\'ia Julia Redondo and Fiorela Rossi Bertone}

\address{Departamento de Matem\'atica, Instituto de Matem\'atica, Universidad Nacional del Sur, Av. Alem 1253, 8000 Bah\'\i a Blanca, Argentina}
\email{mredondo@uns.edu.ar}
\email{fiorela.rossi@uns.edu.ar}

\subjclass[2010]{16S80; 17B55; 18G35.}

\keywords{Maurer-Cartan equation, Hochschild complex, monomial algebras}

\thanks{The first author is a research member of CONICET (Argentina) and has been  supported  by  the  project  PICT-2015-0366.}

\date{\today}

\begin{abstract}
	Let   $A$ be a  monomial associative finite dimensional algebra over a field  $\K$ of characteristic zero. It is well known that the Hochschild cohomology of $A$ can be computed using Bardzell's complex $B(A)$.  The aim of this article is to describe an explict $L_\infty$-structure on $B(A)$ that induces a weak equivalence of  $L_\infty$-algebras between $B(A)$ and the Hochschild complex $C(A)$ of $A$. This allows us to describe the Maurer-Cartan equation in terms of elements of degree $2$ in $B(A)$.  Finally, we make concrete computations when $A$ is a truncated algebra, and we prove that Bardzell's complex for radical square zero algebras is in fact a dg-Lie algebra.
\end{abstract}

\maketitle

\section*{Introduction}\label{sect:introduction}

It is well known that deformation problems can be described by dg-Lie algebras by means of the classical Maurer-Cartan equation modulo gauge equivalence, and that $L_\infty$-algebras are the natural  generalization that appears to avoid the 
rigidity of classical dg-Lie algebras. 
The importance of $L_\infty$-algebras in deformation theory comes from the fact that two different deformation problems are equivalent if the corresponding  dg-Lie algebras are equivalent as $L_\infty$-algebras.

Given an associative algebra $A$, it is well known that the set of equivalence classes of infinitesimal deformations of $A$ is in one to one correspondence with the second Hochschild cohomology group $\HH^2(A)$, see \cite{G1}. However, in order to describe non infinitesimal deformations, we need to consider the shifted Hochschild complex $C(A)[1]$ which, endowed with the Gerstenhaber bracket, admits a structure of dg-Lie algebra. 

Concerning monomial algebras, Bardzell's complex $B(A)$ has shown to be more efficient when dealing with computations of the Hochschild cohomology groups.  Comparison morphisms between $C(A)$ and $B(A)$ have been described explicitly in \cite{RR1}.  It is natural to ask if one can translate the dg-Lie algebra structure from $C(A)[1]$ to $B(A)[1]$ in order to describe deformations of $A$ using $B(A)[1]$ instead of $C(A)[1]$. 

In this article we find explicitly a $L_\infty$-structure for $B(A)[1]$.  The proof relies on the existence of a contraction involving $C(A)$ and $B(A)$.  More precisely, if $C$ is a dg-Lie algebra and  \[ \begin{tikzpicture}
\node (A) at (0,0) {$B$};
\node (B) at (2,0) {$C$};
\draw[->] (A.20) -- node[above] {$G$} (B.160);
\draw[<-] (A.340) --  node[below] {$F$} (B.200);
\path[->] (B) edge [out=0-30, in=0+30, loop]
node[right] {$H$} (B);
\end{tikzpicture} \]
is a contraction 
of complexes, 
the linear maps $l_n$, $n\geq 1$, defined recursively using $H$ in Subsection \ref{subsec:Linf}, give $B$ a $L_\infty$-structure.

It is well known that the structure of a $L_\infty$-algebra on a graded vector space $V$ is equivalent to the structure of a coderivation differential on $\Gamma_c(sV)$, the cofree Com-coalgebra on the suspension of $V$, see \cite{LM,LS}. This characterization is sometimes used as a definition of $L_\infty$-algebras, see e.g. 
\cite[Definition 4.3.]{K}.  This equivalence can be used to transfer $L_\infty$-structures  or $A_\infty$-structures as particular cases of the so-called homotopy transfer theorem or homological perturbation lemma, see \cite{Be,GLS,HK}.  Our approach is different since we just rely on the contraction of complexes to get the explicit formulas instead of using coderivations. 

Even though the $L_\infty$-algebra $B(A)[1]$ is not in general a dg-Lie algebra, under some hypothesis on the monomial algebra $A$, for instance being a truncated algebra, we show that restricting our attemption to elements of degree one, $l_n = 0$ for all $n\geq 3$. This ensures us that, for the calculation of the Maurer-Cartan elements, it suffices to know the bracket $l_2$. 

\

The paper is organized as follows. The first section contains the definitions and basic facts about monomial algebras, Bardzell's complex and the differential graded structure of the shifted Hochschild complex $C(A)[1]$. 
In Section \ref{sec:transfer} we give the recursive formula for the transfer of $L_\infty$-structures in terms of  a contraction of complexes involving a dg-Lie algebra.
In Section \ref{Sec 3} we show that there exists a contraction of complexes involving $C(A)$ and $B(A)$, and we find a recursive formula for the homotopy $H$. Finally in Section \ref{sec:Barzdell} we use the inductive description of the $L_\infty$-structure of the shifted Bardzell's complex $B(A)[1]$ to get some particular results concerning truncated algebras and Maurer-Cartan elements. We present examples that show that this structure is not nilpotent in general.

\section{Preliminaries}\label{sect:preliminaries}

We fix a field $\K$ of characteristic zero as ground field, and $\otimes=\otimes_\K$.

\subsection{Quivers, relations and monomial algebras} 

We briefly recall some concepts concerning quivers and monomial algebras and we refer the reader to \cite{ARS}, for instance, for
unexplained notions.

Let $A = \K Q/I$ be the quotient of a path algebra $\K Q$ for a finite quiver $Q$ and an admissible ideal $I$. As usual, we let $Q_0$ and $Q_1$ be the sets of 
vertices and arrows of the quiver $Q$,  and $s, t : Q_1 \to Q_0$ the maps associating to each arrow $\alpha$ its source  $s(\alpha)$ and its target $t(\alpha)$.  A path $w$ of length $l$ is a sequence of $l$ arrows $\alpha_1 \dots \alpha_l$ such that $t(\alpha_i)=s(\alpha_{i+1})$.

The elements in $I$ are called \textit{relations}, and $\K Q/I$ is called a \textit{monomial algebra} if the ideal $I$ is generated by paths.

By a fundamental result in representation theory  it  is well known that if $A$ is  an associative  finite dimensional algebra over an algebraically closed field $\K$,   there exists a finite quiver $Q$ such that $A$ is Morita equivalent to $\K Q/I$, where $\K Q$ is the path algebra of $Q$ and $I$ is an admissible two-sided ideal of $\K Q$.  The pair $(Q,I)$ is called a \textit{presentation} of $A$.

\subsection{Hochschild and Bardzell's complexes}\label{subsec:complejos} 

The  \textit{bar resolution} of $A$ is the standard free resolution of the $A$-bimodule $A$ given by 
$$C_*(A)= (A \otimes A^{\otimes n} \otimes A, d_n)_{n \geq 0}.$$ The differential of a $n$-chain is the $(n-1)$-chain given by
\begin{align}\label{d_n}
d_n (a_0 \otimes \cdots \otimes  a_{n+1}) =  \sum_{i=0}^{n} (-1)^{i}  a_0 \otimes \cdots \otimes a_{i-1} \otimes & a_i a_{i+1} \otimes  a_{i+2} \otimes \cdots \otimes a_{n+1}
\end{align}
and an homotopy contraction is given by the $\K$-$A$-map
\begin{equation} \label{s_n}
s_n :A \otimes A^{\otimes n} \otimes  A \to A \otimes A^{\otimes n+1} \otimes A
\end{equation}
such that $s_n(x) = 1 \otimes x$ for any $x \in  A \otimes A^{\otimes n} \otimes  A$, that is, 
\begin{equation}
\Id = s_{n-1}d_n + d_{n+1} s_n.
\end{equation}
Applying the functor $\Hom_{A-A}(-,A)$ to this resolution, and using the isomorphism
\begin{align*}
\Hom_{A-A}(A\otimes V \otimes A,A) & \simeq \Hom_\K (V,A), \qquad
\hat f \mapsto  f 
\end{align*}
given by $f (v) = \hat f (1 \otimes v \otimes 1)$,
we get the  \textit{Hochschild complex} $$C^*(A)= ( \Hom_\K (A^{\otimes n}, A),  d^n)_{n \geq 0}$$
whose cohomology is the  \textit{Hochschild cohomology} $\HH(A)$ of $A$ with coefficients in $A$. 
The differential of a $n$-cochain is the $(n+1)$-cochain given by
\begin{align*}
&(-1)^n (d^nf)(a_0 \otimes  \cdots \otimes a_{n}) =   \hat f d_{n+1} (1 \otimes a_0 \otimes  \cdots \otimes a_{n} \otimes 1)  =  a_0 f(a_1 \otimes \cdots \otimes a_n) \\
&\, - \sum_{i=0}^{n-1} (-1)^{i} f(a_0 \otimes \cdots \otimes a_{i-1}\otimes a_i a_{i+1} \otimes a_{i+1} \otimes \cdots \otimes a_n) 
+ (-1)^{n+1} f(a_0 \otimes \cdots \otimes a_{n-1})a_n.
\end{align*}
Since $A$ es $\K$-projective, the Hochschild cohomology groups $\HH^n(A)$ can be identified with the groups $\Ext^n_{A-A}(A,A)$. This means that we can replace the bar resolution by any convenient resolution of the $A$-bimodule $A$.
In the particular case of a monomial algebra $A$,  Hochschild cohomology computations have been mainly made using Bardzell's resolution, see \cite{B},
$$B^*(A) = (\Hom_{E-E} ({\K} AP_n, A),  (-1)^n \delta^n)_{n \geq 0}$$
where $AP_n$ is the set of supports of $n$-concatenations  associated to a presentation $(Q,I)$ of $A$, $E=\K Q_0$ and $\delta^n (f) = \hat f \delta_{n+1}$,  see \cite[Section 2.3]{RR1} for definitions.

\subsection{Gerstenhaber bracket}

Let $f$ be a Hochschild $n$-cochain and $g$ a $m$-cochain. The Gerstenhaber product of $f$ by $g$ is the $(n+m-1)$-cochain defined by
\begin{align*}
f \circ g= \sum_{i=0}^{n-1} (-1)^{i(m+1)} f \circ_i g =\sum_{i=0}^{n-1} (-1)^{i(m+1)} f (\Id^{\otimes i} \otimes g \otimes \Id^{\otimes n-i-1}).
\end{align*}
The Gerstenhaber bracket is defined by
$$[f, g] = f \circ g -  (-1)^{(n-1)(m-1)} g \circ f.$$
In particular, if $n=m=2$,
\begin{align}\label{circ}
[f, g] =  f(g \otimes \Id - \Id \otimes g) +  g(f \otimes \Id - \Id \otimes f).
\end{align}
This bracket satisfies the super Jacobi identity 
$$[[f, g], h] + (-1)^{(\vert f  \vert  -1)(\vert g \vert + \vert h\vert -1)} [[g, h], f] -
(-1)^{(\vert h \vert -1)  (\vert g \vert -1)} [[f, h], g] =0,$$
the Hochschild differential can be expressed in terms of this bracket and the multiplication $\mu$ of $A$ as
$$df  +  [\mu,f]=0$$
and it satisfies
\begin{align*}
d [ f,g] & =  [ d f, g] -   (-1)^{(\vert g \vert-1)(\vert f \vert  -1)}  [ dg, f ] . 
\end{align*}
Hence, the shifted Hochschild complex
$C^*(A)[1]$ endowed with the Gerstenhaber bracket is a dg-Lie algebra.

\subsection{Contractions}
Recall that a contraction in the sense of \cite[page 81]{EM}, referred to as well as SDR-data in the literature (strong deformation retract), is a diagram of complexes
\[ \begin{tikzpicture}
\node (A) at (0,0) {$Y$};
\node (B) at (2,0) {$X$};
\draw[->] (A.20) -- node[above] {$\iota$} (B.160);
\draw[<-] (A.340) --  node[below] {$\pi$} (B.200);
\path[->] (B) edge [out=0-30, in=0+30, loop, swap] 
node[right] {$h$} (B);
\end{tikzpicture} \]
where $\iota$ and $\pi$ are morphisms of complexes and $h$ is a map of degree $-1$ such that
\begin{align} \label{homotopy} 
\pi \iota = \Id, \qquad &\iota \pi - \Id= hd_X + d_Xh, \mbox{ and} \\ \label{ceros}
& \pi h=0, h \iota =0, hh=0.
\end{align}
The original definition of Eilenberg and Mac Lane do not require $hh= 0$; however,
if $h$ satisfies the remaining four conditions, then  $hd_Xh$ satisfies also the fifth.
Moreover, as mentioned in 
\cite[Remark 2.1]{H}, given data satisfying \eqref{homotopy}, condition \eqref{ceros} can be asserted after $h$ has been replaced by $\hat hd_X \hat h$, where 
$\hat h = (\iota \pi- \Id) h (\iota \pi - \Id)$.

\section{Tansfer of $L_\infty$-structures}\label{sec:transfer}

In this section we introduce the notions of $L_\infty$-algebra and weak $L_\infty$-morphism. Then, given a contraction of complexes 
\[ \begin{tikzpicture}
\node (A) at (0,0) {$B$};
\node (B) at (2,0) {$C$};
\draw[->] (A.20) --  (B.160);
\draw[<-] (A.340) --   (B.200);
\path[->] (B) edge [out=0-30, in=0+30, loop]
(B);
\end{tikzpicture} \]
where $C$ has structure of dg-Lie algebra, we describe explicitly a $L_\infty$-structure on $B$ and a weak $L_\infty$-morphism $\phi: B \to C$. For this we need a series of technical lemmas.

\subsection{$L_\infty$-algebras}
Let $V$ be a $\Z$-graded vector space. Denote by $\wedge V$ the free graded commutative associative algebra over $V$. Let $v_1,\dots,v_n\in V$ be homogeneous elements and $v_1\wedge \dots \wedge v_n$ be its product in $\wedge V$.

\begin{definition}
	Let $\sigma\in \s_n$ be a permutation and $v_1,\dots,v_n\in V$ homogeneous. The Koszul sign $\varepsilon(\sigma) = \varepsilon(\sigma; v_1,\dots,v_n)\in\{\pm1\}$ is defined by
	$$v_1\wedge \dots \wedge v_n= \varepsilon(\sigma)v_{\sigma(1)}\wedge \dots \wedge v_{\sigma(n)},$$
	the antisymmetric Koszul sign $\chi(\sigma) = \chi (\sigma; v_1,\dots,v_n) \in \{\pm1\}$ is defined by 
	$$\chi(\sigma) =\sgn(\sigma)\varepsilon(\sigma; v_1, \cdots, v_n)$$ 
	and, for any $t$ with {$1 \leq t < n$, $\kappa (\sigma)_{t} =  \kappa(\sigma, t;  v_1,\dots,v_n)\in\{\pm1\}$ is given by
		$$\kappa (\sigma)_{t} = (-1)^{(t-1) + (n-t-1)( \sum_{p=1}^t \vert v_{\sigma(p)} \vert )}.$$}
\end{definition}

Let $\s_{t,n-t}$ be the set of all $(t,n-t)$-unshuffles in the symmetric group $\s_n$, that is, 
\[ \s_{t, n-t} = \{ \sigma \in \s_n : \sigma(1) < \cdots < \sigma (t), \ \sigma(t+1) < \cdots < \sigma (n) \} .\] 
Let $\s^{-}_{t,n-t}$ be the set of all $ \sigma \in \s_{t,n-t}$ such that $\sigma(1) < \sigma(t+1)$.  Analogously, for $i+j+k=n$ we define
\[ \s_{i,j,k} = \{ \sigma \in \s_n : \sigma(1) < \cdots < \sigma (i), \ \sigma(i+1) < \cdots < \sigma (i+j), \ \sigma(i+j+1) < \cdots < \sigma (n) \}\] 
and $\s^{-}_{i,j,k}$ the set of all $\sigma\in\s_{i,j,k}$ such that $\sigma (1) < \sigma (i+1) < \sigma (i+j+1)$.
We specify permutations by the list of their values, that is, for any $\sigma \in \s_n$ we write $\sigma = (\sigma(1), \cdots, \sigma(n))$.

\begin{definition}
	A $L_\infty$-algebra is a $\Z$-graded vector space $L$ together with linear maps
	$$l_n: \otimes^n L \to L$$
	of degree {$2-n$} such that, for every $n \in \N$ and homogeneous $v_1, \cdots, v_n \in L$,  the following conditions are  satisfied:
	\begin{align}
	& l_n \  \hat \sigma = \chi(\sigma)\  l_n, \qquad \forall \, \sigma \in \s_n;
	\\
	& \sum_{i+j=n+1}  \sum_{\sigma \in \s_{i,n-i}} (-1)^{i {(j-1)}}  \chi(\sigma)  \
	l_j ( l_i \otimes \Id^{\otimes n-i}) \ \hat \sigma =0,
	\end{align}
	where $\hat \sigma (v_1 \otimes \cdots  \otimes  v_n) = v_{\sigma(1)} \otimes  \cdots  \otimes v_{\sigma(n)}$.
\end{definition}

\begin{remark}
	Observe that for $n=2$, the equations in the previous definition are
	\begin{align} \label{eq: conmutar}
	l_2 (v_{2} \otimes v_{1}) &= - (-1)^{\vert v_1\vert \vert v_2 \vert} l_2 (v_1 \otimes v_2)\\ \label{eq: corchete1}
	l_1 l_2  (v_1 \otimes v_2) & = l_2 (l_1(v_1) \otimes v_2) -  (-1)^{\vert v_1\vert \vert v_2 \vert} l_2 (l_1(v_2) \otimes v_1) \\ \notag
	&= l_2 (l_1(v_1) \otimes v_2) + (-1)^{\vert v_1\vert} l_2 (v_1 \otimes l_1(v_2)).
	\end{align}
	Moreover, when $n=3$ and $l_3=0$, we get the Jacobi identity
	\begin{align} \label{jacobi}
	\sum_{\sigma \in \s_{2,1}} \chi(\sigma) l_2 (l_2 \otimes \Id) \hat \sigma =0.
	\end{align}
\end{remark}

\begin{definition} \cite[Def.~5.2]{LM}
	Let $L = (L, l_n)$ be a $L_\infty$-algebra and $B = (B, d , [- , -])$ a dg-Lie algebra. A {weak} $L_\infty$-morphism from $L$ to $B$ is a collection of skew symmetric linear maps $\phi_n : \otimes^n L \to B$ of degree $1-n$ such that, for every $n \in \N$ and homogeneous $v_1, \cdots, v_n \in L$,
	\begin{align*}
	d \phi_n  + \sum_{j+k=n+1} \sum_{\sigma \in \s_{k,n-k}} & \chi(\sigma) 
	(-1)^{k{ (j-1)}+1}
	\phi_j ( l_k \otimes \Id^{\otimes n-k} ) \ \hat \sigma \\
	&+ \sum_{s+t=n} \sum_{\tau \in \s^{-}_{s,n-s}} \chi(\tau) \kappa(\tau)_{s}
	[ \phi_s, \phi_t ]  \ \hat \tau =0
	\end{align*}
	where  
	$ [ \phi_s, \phi_t]  \ \hat \tau (v_1 \otimes \cdots \otimes v_n) =  [ \phi_s(v_{\tau(1)} \otimes \cdots \otimes v_{\tau(s)}), \phi_t(v_{\tau(s+1)} \otimes \cdots \otimes v_{\tau(n)})] $.
\end{definition}

\subsection{$L_\infty$-structure}\label{subsec:Linf}

Let $C=(C^n, d^n, [- , -])$ be a dg-Lie algebra, $B=(B^n, \delta^n)$ be a cochain complex,  and
\[ \begin{tikzpicture}
\node (A) at (0,0) {$B$};
\node (B) at (2,0) {$C$};
\draw[->] (A.20) -- node[above] {$G^*$} (B.160);
\draw[<-] (A.340) --  node[below] {$F^*$} (B.200);
\path[->] (B) edge [out=0-30, in=0+30, loop]
node[right] {$H^*$} (B);
\end{tikzpicture} \]
where $F^*, G^*$ are morphisms of complexes, $H^*$ is a map of degree $-1$ such that
\begin{align} \nonumber
F^* G^*  = \Id, \qquad &G^* F^* - \Id= H^* d + d H^*, \mbox{ and} \\  \label{vanish}
& F^* H^*=0.
\end{align}
We define recursively the linear maps of degree $2-n$ 
$$l_n: \otimes^n B\to B, \qquad u_n,v_n: \otimes^n B\to C$$ 
and the linear maps of degree $1-n$ 
$$\phi_n: \otimes^n B\to C$$ 
by
$
l_1=\delta^*$, $v_1 =  u_1=0$, $\phi_1= G^*$;
and, for $n >1$ and homogeneous $f_1, \dots, f_n \in B$, 
\begin{align}\label{eq: vn}
v_n  & =  \sum_{t=1}^{n-1} \sum_{\tau \in \s^{-}_{t,n-t}} \chi(\tau){ \kappa (\tau)_{t}} \  [ \phi_t , \phi_{n-t}] \hat \tau  \\  \label{eq: ln}
l_n  & =   F^* v_n, \\ 
\label{eq: un}
u_n & = \sum_{k=2}^n \sum_{\tau \in \s_{k,n-k}} 
(-1)^{k(n-k)+1}  \chi(\tau) \
\phi_{n-k+1} ( l_k \otimes \Id^{\otimes n-k}) \hat \tau,\\ 
\label{eq: phin}
\phi_n & = H^* (u_n + v_n)
\end{align}
where
\begin{align*}
\hat \tau (f_1 \otimes  \cdots \otimes f_n) & = f_{\tau(1)} \otimes \cdots \otimes f_{\tau(n)}\ , \mbox{and}  \\
[ \phi_t , \phi_{n-t}] (f_1 \otimes \cdots \otimes f_n) & = [ \phi_t (f_1 \otimes \cdots \otimes f_t), \phi_{n-t}(f_{t+1} \otimes \cdots \otimes f_n)].
\end{align*}
Observe that the assumption $F^*H^*=0$ implies that $F^*\phi_n=0$ for all $n>1$. Hence 
\begin{align} \label{GF}
F^* u_n = -  F^*\phi_1 l_n  = - F^*G^* l_n =  - F^* v_n.
\end{align}

\begin{remark}\label{solo v}
	Observe that if $H^*G^*=0$ and $H^*H^*=0$ then $H^*u_n=0$ and hence $\phi_n = H^* v_n$. However, we need the auxiliary maps $u_n$ in the following proofs, therefore we proceed with the weaker assumption as stated in \ref{vanish}.
\end{remark}

\begin{lemma}
	For $n\geq 2$ the maps $u_n$, $v_n$, $l_n$ and $\phi_n$  defined above are skew symmetric. 
\end{lemma}
\begin{proof}
	We prove the lemma by induction. For $n=2$ the lemma follows immediately. We assume the claim for every map and every number lower than $n$. 
	Consider first the map $v_n$. Let $\psi\in\s_n$. For every $1\leq t<n$ and $\tau \in\s^{-}_{t,n-t}$ consider the permutation $\sigma \in\s^{-}_{t,n-t}$ such that there exist $\nu_1\in\s_t$ and $\nu_2\in\s_{n-t}$ satisfying 
	$$\hat\tau \hat\psi=(\hat\nu_1\otimes \hat\nu_2)\hat{\sigma}.$$
	It follows immediately that $\kappa(\tau)_{t}=\kappa(\sigma)_{t}$ and $\chi(\tau)\chi(\psi)=\chi(\sigma)\chi(\nu_1)\chi(\nu_2)$. Hence, by the skew symmetric property of $\phi_t$, $\phi_{n-t}$ and the bracket  $[-,-]$, one can check that $v_n$ is skew symmetric. Now, the property for $l_n$  follows by definition.
	
	Following the same argument as before, for every $2\leq k\leq n$ and $\tau\in\s_{k,n-k}$ consider the permutation $\sigma\in\s_{k,n-k}$ such that there exist $\nu_1\in\s_k$ and $\nu_2\in\s_{n-k}$ satisfying
	$$\hat\tau\hat\psi=(\hat\nu_1\otimes \hat\nu_2)\hat\sigma.$$
	Hence, since $\chi(\tau)\chi(\psi)=\chi(\sigma)\chi(\nu_1)\chi(\nu_2)$ and the maps $\phi_{n-k+1}$ and $l_k$ are skew symmetric, the map $u_n$ is so. By the definition of $\phi_n$, the proof is done.
\end{proof}

\begin{theorem} \label{infinito} With the above notation, the maps  $l_n:\otimes^n B\to B$, $n\in \N$, give $B$  a $L_\infty$-structure. Moreover, this structure extends $G$ to a weak $L_\infty$-morphism $\phi: B \to C$, where $\phi=(\phi_n)$ is defined as above.
\end{theorem}

The proof of  Theorem \ref{infinito} is an immediate consequence of \eqref{eq:l_n} and \eqref{eq:phi_n} in Lemma \ref{induccion}.  Before proving this lemma, we need 
three technical lemmas that show the connection between the maps $u_n, v_n, \phi_n$ and $l_n$.

\begin{lemma}\label{lem:phi,v} 
	For all $n \geq 2$, the maps $v_n$ and $\phi_n$ satisfy the equation
	\[ \sum_{t=2}^{n-1}  \sum_{\tau \in \s_{t,n-t}}\chi(\tau)  \kappa(\tau)_{t}  [v_t, \phi_{n-t}] \hat \tau = 0.\]
\end{lemma}
\begin{proof}
	By definition of $v_t$,
	\begin{align} \label{cuenta}
	\sum_{t=2}^{n-1}  &\sum_{\tau \in \s_{t,n-t}} \chi(\tau) \kappa (\tau)_{t}  [v_t, \phi_{n-t}]\hat \tau \\ \notag
	=&\sum_{t=2}^{n-1}  \sum_{\tau \in \s_{t,n-t}} \chi(\tau)  \kappa (\tau)_{t}  \sum_{i=1}^{t-1} \sum_{\mu \in \s^{-}_{i, t-i}} \chi(\mu) {\kappa (\mu)_{i}} [ [\phi_i, \phi_{t-i}]\hat \mu,\phi_{n-t} ]   \hat \tau. 
	\end{align} 
	It is clear that, for any $\tau \in \s_{t,n-t}, \mu \in \s^{-}_{i, t-i}$,
	$ (\hat \mu \otimes \Id^{\otimes n-t}) \hat \tau = \hat \sigma$
	for some $\sigma \in \s_{i, t-i, n-t}$. 
	Set
	\begin{align*}
	w_i (f_1 \otimes \cdots \otimes f_n) & = \phi_i(f_{\sigma(1)} \otimes \dots \otimes  f_{\sigma(i)}), \\
	w_{t-i} (f_1 \otimes \cdots \otimes f_n)  & = \phi_{t-i}(f_{\sigma(i+1)} \otimes \dots \otimes  f_{\sigma(t)}), \\
	w_{n-t} (f_1 \otimes \cdots \otimes f_n) & = \phi_{n-t}(f_{\sigma(t+1)} \otimes \dots \otimes f_{\sigma(n)}).
	\end{align*}
	Then \eqref{cuenta} equals
	\begin{align*}
	\sum_{t=2}^{n-1}  \sum_{i=1}^{t-1}  \sum_{\sigma \in \s^{-}_{i,t-i,n-t}} \sum_{\psi \in \s_{2,1}} \varepsilon^\psi [[w_i, w_{t-i}], w_{n-t}] \hat \psi
	\end{align*}
	where $\varepsilon^\psi\in\{\pm1\}$.  A direct but tedious computation shows that $\varepsilon^\psi = \varepsilon^\Id \chi(\psi)$. Hence the lemma follows by Jacobi identity, see \eqref{jacobi}.
\end{proof} 

\begin{lemma}\label{lem:phi-u} 
	For all $n \geq 2$, the maps $u_n$ and $\phi_n$ satisfy the equation
	\begin{align*}
	\sum_{t=2}^{n-1} \sum_{\sigma \in \s_{t,n-t}} \chi(\sigma)
	{\kappa(\sigma)_{t}} & [u_t ,\phi_{n-t}] \hat \sigma 
	=  \sum_{k=2}^{ n-1} \sum_{\psi \in \s_{k,n-k}}    (-1)^{k(n-k)} \chi(\psi) \
	v_{n-k+1} ( l_k \otimes \Id^{\otimes n-k}) \hat \psi.
	\end{align*}
	
\end{lemma}
\begin{proof}
	Using the definition of $u_n$, for any ${\sigma \in \s_{t,n-t}}$ we have
	\begin{align*}
	[u_t, \phi_{n-t}]  \hat \sigma   =   \sum_{k=2}^t \sum_{\tau \in \s_{k, t-k}} { (-1)^{k(t-k)+1}} \chi(\tau)  \ [\phi_{t-k+1}, \phi_{n-t}]  
	((l_k \otimes \Id^{\otimes {t-k}}) \hat \tau \otimes \Id^{\otimes n-t}) \hat \sigma.
	\end{align*}
	Now 
	\[ ((l_k \otimes \Id^{\otimes {t-k}}) \hat \tau \otimes \Id^{\otimes n-t}) \hat \sigma = \hat \mu (l_k  \otimes \Id^{\otimes n-k}) \hat \psi \]
	with ${\psi \in \s_{k, n-k}}, {\mu \in \s^{-}_{t-k+1,n-t}}$. Moreover, $\mu(1)=1$ and
	\begin{align*} 
	\hat \mu (l_k  \otimes \Id^{n-k}) \hat \psi   &(f_1 \otimes \cdots \otimes  f_n)\\ &= l_k(f_{\psi(1)} \otimes \cdots \otimes  f_{\psi(k)}) \otimes f_{\psi(\tau(2)+k-1)}\otimes \cdots
	\otimes f_{\psi(\tau(n-k+1)+k-1)} .	
	\end{align*}
	Since $\chi(\sigma)\chi(\tau)=\chi(\psi)\chi(\mu)$ and 
	$ (-1)^{k(n-k)}\kappa(\mu)_{t-k+1}=(-1)^{k(t-k)+1}\kappa(\sigma)_{t}$, we have
	\begin{align*}
	\sum_{t=2}^{ n-1}  &\sum_{\sigma \in \s_{t,n-t}}  \chi(\sigma)\kappa(\sigma)_{t}  [u_t, \phi_{n-t}]  \hat \sigma \\
	&= \sum_{k=2}^{n-1} \sum_{t=k}^{n-1} \sum_{\psi \in \s_{k, n-k}} \sum_{\mu \in \s^{-}_{t-k+1,n-t}} { \chi(\psi)\chi(\mu)\kappa(\mu)_{t-k+1}} [\phi_{t-k+1}, \phi_{n-t}]  \hat \mu (l_k  \otimes \Id^{\otimes n-k}) \hat \psi . 
	\end{align*}
	Hence the lemma follows by the definition of $v_{n-k+1}$.
\end{proof}

\begin{lemma} \label{lem: u,phi,l}
	For  $n\geq 2$, the following equality holds 
	\begin{align*} 
	\sum_{k=2}^{n-1} &\sum_{\sigma \in \s_{k,n-k}}  (-1)^{{k(n-k)+1}}  \chi(\sigma) 
	u_{n-k+1} ( l_k \otimes \Id^{\otimes n-k}) \hat \sigma \\ 
	& =  {\sum_{k=3}^{n}  \sum_{i=2}^{k-1}  \sum_{\mu \in \s_{i,k-i,n-k}} \hspace{-0.3cm} (-1)^{n(k-1)+i(k-i)} \chi({\mu})  
		\phi_{n-k+1} (l_{k-i+1} \otimes \Id^{\otimes n-k} )(l_i \otimes \Id^{\otimes n-i}) \hat \mu}.
	\end{align*}
\end{lemma}

\begin{proof}
	We have
	\begin{align*} 
	\sum_{k=2}^{n-1}& \sum_{\sigma \in \s_{k,n-k}} (-1)^{k(n-k)+1}  \chi(\sigma) 
	\ u_{n-k+1} ( l_k \otimes \Id^{\otimes n-k}) \hat \sigma \\ 
	& = \sum_{t=2}^{n-1} \sum_{\sigma \in \s_{n-t+1,t-1}}
	(-1)^{(t-1)(n-t+1)+1}  \chi(\sigma)  \
	u_{t} ( l_{n-t+1} \otimes \Id^{\otimes t-1}) \hat \sigma
	\end{align*}	
	and
	\begin{align*} 
	u_{t} ( l_{n-t+1} \otimes \Id^{\otimes t-1}) 
	=  \sum_{j=2}^{t} \sum_{\tau \in \s_{j, t-j}}  (-1)^{j(t-j)+1} \chi(\tau)\ \phi_{t-j+1} (l_j \otimes \Id^{t-j} ) \hat \tau ( l_{n-t+1} \otimes \Id^{\otimes t-1})
	\end{align*}
	where $\tau(1) = 1$ or $\tau(1) = j+1$. If $\tau(1)=1$ then
	$$(l_j \otimes \Id^{t-j} ) \hat \tau ( l_{n-t+1} \otimes \Id^{\otimes t-1}) \hat \sigma 
	= ( l_j ( l_{n-t+1} \otimes \Id^{\otimes j-1}) 
	\otimes \Id^{\otimes t-j}) \hat \psi $$
	for some $\psi \in \s_{ n-t+1, j-1, t-j}$. On the other hand, if $\tau(1)=j+1$ then
	$$(l_j \otimes \Id^{t-j} ) \hat \tau ( l_{n-t+1} \otimes \Id^{\otimes t-1}) \hat \sigma 
	= ( l_j \otimes l_{n-t+1} \otimes \Id^{\otimes t- j-1})  \hat \psi $$
	for some $\psi \in \s_{j, n-t+1, t-j-1}$. Since $\phi_{t-j+1}$ is skew symmetric, one can observe that this kind of summands appear twice and with different signs, hence they cancel each other.
	Then
	\begin{align*} 
	\sum_{k=2}^{n-1}& \sum_{\sigma \in \s_{k,n-k}} 
	(-1)^{k(n-k)+1}   \chi(\sigma) \
	u_{n-k+1} ( l_k \otimes \Id^{\otimes n-k}) \hat \sigma \\ 
	& =  \sum_{t=2}^{n-1} \sum_{j=2}^{t} \sum_{\psi \in \s_{n-t+1, j-1, t-j}} 
	\hspace{-0.5cm} (-1)^{(t-1)(n-t+1)+j(t-j)}  \chi(\psi)   \phi_{t-j+1} (l_j \otimes \Id^{t-j} )  ( l_{n-t+1} \otimes \Id^{\otimes t-1}) \hat \psi 
	\end{align*}
	and, changing parameters by the rule $(k,i) = (j+i-1, n-t+1)$, the proof is done.
\end{proof}

The proof of Theorem \ref{infinito} follows immediately by  statements \eqref{eq:l_n} and \eqref{eq:phi_n} of the next lemma.

\begin{lemma} \label{induccion}
	For  $n\geq 2$, the following equalities hold
	\begin{align}   \label{eq:v_n}
	dv_n = &   \sum_{\mu\in\s_{1,n-1}}  { (-1)^n}  \chi(\mu) \ v_n (\delta \otimes \Id^{\otimes n-1})\hat\mu 
	- \sum_{t=2}^{n-1} \sum_{\tau \in \s_{t,n-t}} 
	{ \chi(\tau) \kappa(\tau)_{t} [u_t, \phi_{n-t}] }
	\hat \tau;  \\ \label{eq:l_n}
	\delta  l_n  =&
	{ \sum_{i=1}^{ n-1} }\sum_{\sigma \in \s_{i,n-i}} { (-1)^{i(n-i)+1}}  \chi(\sigma) \
	l_{n-i+1} ( l_i \otimes \Id^{\otimes n-i}) \hat \sigma; \\
	\label{eq:u_n}
	du_n   = & { \sum_{\mu\in\s_{1,n-1}}  { (-1)^n} \chi(\mu) \ u_n (\delta \otimes \Id^{\otimes n-1})\hat\mu } \\ \nonumber
	& + \sum_{k=2}^{n-1} \sum_{\sigma \in \s_{k,n-k}} 
	(-1)^{{k(n-k)}} \chi(\sigma) \ 
	(u_{n-k+1}+v_{n-k+1})  ( l_k \otimes \Id^{\otimes n-k}) \hat \sigma \\ \nonumber
	&  +  \sum_{{ k=3}}^{ n}  \sum_{i=2}^{k-1} 
	\sum_{\psi \in \s_{i,k-i,n-k}}  { (-1)^{n(k-1)+i(k-i)}}\chi({\psi}) \
	\phi_{n-k+1} (l_{k-i+1} \otimes \Id^{\otimes n-k} )(l_i \otimes \Id^{\otimes n-i}) \hat \psi; \\ \label{eq:phi_n}
	d\phi_{n}  = &  { \sum_{\mu\in\s_{1,n-1}}   { (-1)^{n+1}} \chi(\mu) \ \phi_n (\delta \otimes \Id^{\otimes n-1})\hat\mu } - u_n - v_n.
	\end{align}
\end{lemma}	
\begin{proof}
	Notice that, for $n\geq 2$,
	\begin{align*}
	d \phi_n = d H (u_n + v_n) = (- Hd + GF - \Id) (u_n + v_n) = - Hd(u_n + v_n) - u_n - v_n
	\end{align*}	
	since $GF (u_n + v_n) =0$, and that \eqref{eq:phi_n} holds also for $n=1$ since $dG=G\delta$.
	
	We prove all the statements simultaneously by induction. For $n=2$,
	\[ 
	l_2  = F[G,G], \ v_2  =[G,G], \ u_2= -  GF [G,G] \mbox{ and } \phi_2  = H (u_2+v_2).  \]
	Observe first that 
	$\displaystyle v_2 = d [G,G] = { \sum_{\mu\in\s_2} { \chi(\mu)}[G,G](\delta \otimes \Id )\hat\mu} $
	since
	\begin{align*}
	d [G,G] (f_1 \otimes f_2) &= d [ G(f_1), G(f_2) ] =  [ dG(f_1), G(f_2) ] -  (-1)^{|f_1| |f_2|}  [ dG(f_2), G(f_1) ] \\
	& =  [ G\delta (f_1), G(f_2) ]  -  (-1)^{|f_1| |f_2|}   [ G\delta (f_2), G(f_1) ]   \\
	& = 	{ \sum_{\mu\in\s_{1,1}}{ \chi(\mu)} [G,G](\delta \otimes \Id )\hat\mu} (f_1 \otimes f_2). 
	\end{align*}
	Hence $dv_2 = \sum_{\mu\in\s_{1,1}} { \chi(\mu)} v_2(\delta \otimes \Id )\hat\mu$ 
	and, since $u_2 = -GF v_2$ and $dGF=GFd$, the same equality holds for $u_2$.	Now 
	\begin{align*}
	d \phi_{2} & = -Hd (u_2+v_2)  - u_2 - v_2 \\
	& = -  \sum_{\mu\in\s_{1,1}} { \chi(\mu)} H(u_2+v_2)(\delta \otimes \Id )\hat\mu - u_2 - v_2  = -  \sum_{\mu\in\s_{1,1}} { \chi(\mu)} \phi_2 (\delta \otimes \Id )\hat\mu - u_2 - v_2;\\
	\delta l_2 &= \delta F v_2 = F d v_2 = F 	{ \sum_{\mu\in\s_{1,1}} { \chi(\mu)} v_2(\delta \otimes \Id )\hat\mu} = 	{ \sum_{\mu\in\s_{1,1}}  { \chi(\mu)} l_2({ l_1} \otimes \Id )\hat\mu}.
	\end{align*} 
	Therefore the four equations hold for $n=2$.
	Let $n >2$ and assume that all the equations hold for any $m$ with $2 \leq m <n$. 
	We will prove  the formula for $dv_n$. 
	Using equation \eqref{eq: corchete1}, for  
	$\mu= (t+1, \cdots, n, 1, \cdots , t)$ we get that
	\begin{align*}
	dv_n =
	& \sum_{t=1}^{n-1} \sum_{\tau \in \s^{-}_{t,n-t}}  \chi(\tau)  \kappa(\tau)_{ t}  ( [ d\phi_t , \phi_{n-t}] \hat \tau - (-1)^{a.b}[ d \phi_{n-t}, \phi_t] )\hat \mu \hat \tau
	\end{align*}
	where $a = \vert \phi_t(v_{\tau(1)} \otimes \cdots \otimes v_{\tau(t)})\vert$ and $b = \vert \phi_{n-t}(v_{\tau(t+1)} \otimes \cdots \otimes  v_{\tau(n)})\vert$.  Since
	$$\s_{n-t,t} \setminus \s^{-}_{n-t,t} = \{ \mu \tau: \tau \in \s^{-}_{t,n-t} \}$$
	a meticulous study of the signs shows that  
	\begin{align*}
	dv_n = & \sum_{t=1}^{n-1} \sum_{\tau \in \s_{t,n-t}}  \chi(\tau)  \kappa(\tau)_{t}  [ d\phi_t , \phi_{n-t}] \hat \tau.
	\end{align*}
	Now, using the inductive hypothesis on \eqref{eq:phi_n} and that $u_1+v_1=0$, we get that 
	\begin{align*}
	dv_n = & \sum_{t=1}^{n-1} \sum_{\tau \in \s_{t,n-t}} (-1)^{t+1} \chi(\tau)  \kappa(\tau)_{t} \sum_{\mu \in \s_{1,t-1}} \chi(\mu)
	[ \phi_t , \phi_{n-t}]  ((\delta \otimes \Id^{\otimes t-1})\hat \mu\otimes \Id^{\otimes n-t}) \hat \tau \\
	& -  \sum_{t=2}^{n-1} \sum_{\tau \in \s_{t,n-t}}  \chi(\tau)  \kappa(\tau)_{t} [ u_t + v_t, \phi_{n-t}] \hat \tau.
	\end{align*}
	Observe that for any $\tau \in \s_{t,n-t}$ and  $\mu \in \s_{1,t-1}$ we have  
	$(\hat \mu\otimes \Id^{\otimes n-t}) \hat \tau = \hat {\psi } \hat { \nu}$,
	where $\psi \in \s^{-}_{t,n-t}$,  $\nu \in \s_{1,n-1}$, $\nu(1)=\tau (\mu(1))$, and 
	$ \chi(\mu) \chi(\tau)\kappa(\tau)_{t} =\chi(\nu) \chi(\psi) \kappa(\psi)_{ t}$.
	Then
	\begin{align*}
	dv_n = &  \sum_{t=1}^{n-1} \sum_{\psi \in \s^{-}_{t,n-t}}\sum_{\nu \in \s_{1,n-1}} (-1)^{t+1}1\chi(\nu)  \chi(\psi)  \kappa(\psi)_{t}
	[ \phi_t , \phi_{n-t}] (\delta \otimes \Id^{\otimes n-1})  \hat \psi  \hat \nu \\
	& -  \sum_{t=2}^{n-1} \sum_{\tau \in \s_{t,n-t}}  \chi(\tau)  \kappa(\tau)_{t} [ u_t + v_t , \phi_{n-t}] \hat \tau
	\end{align*} 
	and,  since $\kappa_{t} (\psi) (\delta \otimes \Id^{\otimes n-1})  \hat \psi = (-1)^{n-t-1} \kappa_{t} (\psi) \hat \psi  (\delta \otimes \Id^{\otimes n-1})$, we get
	\begin{align*}
	dv_n= & (-1)^n \sum_{t=1}^{n-1} \sum_{\nu \in \s_{1,n-1}} \chi(\nu)   v_n (\delta \otimes \Id^{\otimes n-1})\hat \nu -  \sum_{t=2}^{n-1} \sum_{\tau \in \s_{t,n-t}}  \chi(\tau)  \kappa(\tau)_{t} [ u_t + v_t, \phi_{n-t}] \hat \tau.
	\end{align*} 
	From Lemma \ref{lem:phi,v} we conclude that the formula for $dv_n$ follows for any integer $n \geq 2$. This fact and Lemma \ref{lem:phi-u} allows us to prove \eqref{eq:l_n} for all $n \geq 2$ since
	\begin{align*}
	\delta l_n & =
	F d v_n
	= {  \sum_{\mu\in\s_{1,n-1}}(-1)^n  \chi(\mu) \ F v_n (\delta \otimes \Id^{\otimes n-1})\hat\mu } 
	-  \sum_{t=2}^{n-1} \sum_{\tau \in \s_{t,n-t}} \chi(\tau)
	{\kappa(\tau)_{t}}F [u_t ,\phi_{n-t}] \hat \tau \\
	=& {\sum_{\mu\in\s_{1,n-1}} (-1)^n \chi(\mu)\  l_n (\delta \otimes \Id^{\otimes n-1})\hat\mu } 
	+   \sum_{t=2}^{ n-1}  \sum_{\sigma \in \s_{t,n-t}}  { (-1)^{t(n-t)+1}} \chi(\sigma) \
	l_{n-t+1} ( l_t \otimes \Id^{\otimes n-t}) \hat \sigma \\
	&=  \sum_{i=1}^{ n-1}  \sum_{\sigma \in \s_{i,n-i}}  { (-1)^{i(n-i)+1}} \chi(\sigma) \
	l_{n-i+1} ( l_i \otimes \Id^{\otimes n-i}) \hat \sigma.
	\end{align*}
	Using the inductive hypothesis on \eqref{eq:phi_n},  since $u_1 + v_1=0$, we get that 
	\begin{align*}
	du_n  =&  \sum_{k=2}^n \sum_{\sigma \in \s_{k,n-k}} 
	(-1)^{k(n-k)+1} \chi(\sigma)  \
	d\phi_{n-k+1} ( l_k \otimes \Id^{\otimes n-k}) \hat \sigma\\
	=&\sum_{k=2}^n \sum_{\sigma \in \s_{k,n-k}} 
	(-1)^{n(k-1)+1}  \chi(\sigma)  \ \sum_{\mu \in \s_{1,n-k}} \chi(\mu)
	\phi_{n-k+1} (\delta\otimes \Id^{n-k})\hat \mu (l_k \otimes \Id^{\otimes n-k} )\hat \sigma \\
	& + \sum_{k=2}^{n-1} \sum_{\sigma \in \s_{k,n-k}} 
	(-1)^{k(n-k)}  \chi(\sigma) \
	(u_{n-k+1} + v_{n-k+1} ) ( l_k \otimes \Id^{\otimes n-k}) \hat \sigma.
	\end{align*}
	Consider the summand corresponding to $\mu= \Id \in \s_{1,n-k}$. Since \eqref{eq:l_n} holds for  $n \geq 2$, we get
	\begin{align*}
	\sum_{k=2}^n& \sum_{\sigma \in \s_{k,n-k}} 
	(-1)^{n(k-1)+1} \chi(\sigma) 
	\phi_{n-k+1} (\delta l_k \otimes \Id^{\otimes n-k}) \hat \sigma \\
	= &  \sum_{k=2}^n \sum_{\sigma \in \s_{k,n-k}} 
	\chi(\sigma)  \ 
	{ \sum_{i=1}^{ k-1} }  \sum_{\tau \in \s_{i,k-i}}  
	(-1)^{(n+i)(k-1)} \chi(\tau) \
	\phi_{n-k+1} (   l_{k-i+1} ( l_i \otimes \Id^{\otimes k-i}) \hat \tau \otimes \Id^{\otimes n-k}) \hat \sigma  \\	
	= &  \sum_{k=2}^n \sum_{\sigma \in \s_{k,n-k}}  \sum_{\tau \in \s_{1,k-1}}  
	{(-1)^{(n+1)(k-1)}} \chi(\sigma)   \chi(\tau) \
	\phi_{n-k+1} (   l_{k} ( \delta \otimes \Id^{\otimes k-1}) \hat \tau \otimes \Id^{\otimes n-k}) \hat \sigma  \\	
	&  +  \sum_{k=2}^n \sum_{i=2}^{ k-1}   \sum_{\sigma \in \s_{k,n-k}}  \sum_{\tau \in \s_{i,k-i}} \hspace{-0.2cm} (-1)^{(n+i)(k-1)} \chi(\sigma) 
	\chi(\tau) \ \phi_{n-k+1} (   l_{k-i+1} ( l_i \otimes \Id^{\otimes k-i}) \hat \tau \otimes \Id^{\otimes n-k}) \hat \sigma  \\
	= &  \sum_{k=2}^n   \sum_{\nu \in \s_{1,n-1}}   \sum_{\psi \in \s_{k,n-k},  \psi(1)=1} 
	{(-1)^{(n+1)(k-1)}} \chi(\psi)   \chi(\nu) \
	\phi_{n-k+1} ( l_{k} \otimes \Id^{\otimes n-k}) \hat \psi  (\delta \otimes \Id^{\otimes n-1}) \hat \nu  \\	
	& +  \sum_{k=3}^n \sum_{i=2}^{ k-1}   \sum_{\psi \in \s_{i,k-i,n-k}}  {(-1)^{n(k-1)+i(k-i)}} \chi(\psi) 
	\ \phi_{n-k+1}  ( l_{k-i+1} \otimes \Id^{\otimes n-k})( l_i \otimes \Id^{\otimes n-i}) \hat \psi.		
	\end{align*} 
	On the other hand, when $\mu \in \s_{1,n-k}$, $\mu \not = \Id$, and $\sigma \in \s_{k,n-k}$,
	\[ \widehat{ \tau} (\delta\otimes \Id^{n-k})\hat \mu (l_k \otimes \Id^{\otimes n-k} )\hat \sigma =( l_k \otimes \Id^{n-k}) \hat \psi (\delta \otimes \Id^{n-1}) \hat \nu\]
	for some $\nu \in \s_{1,n-1}, \psi\in \s_{k, n-k}$ with $\psi(1)=k+1$, and $\tau=(2,1)$. Since $\phi_{n-k+1}$ is skew symmetric, by the definition of $u_n$, we finally get the desired formula for $du_n$.
	
	Finally, since \eqref{eq:v_n} and  \eqref{eq:u_n} hold for all $n \geq 2$, we have that
	\begin{align*}
	d  \phi_{n} = &  - u_n - v_n  - H  d(v_n + u_n) \\
	= & - u_n - v_n - (-1)^n \hspace{-0.2cm} \sum_{\mu \in \s_{1,n-1}}\hspace{-0.3cm} H(v_n + u_n)  ( \delta \otimes \Id^{\otimes n-1})  \hat \mu  
	+   \sum_{t=1}^{n-1} \sum_{\tau \in \s_{t,n-t}}  \hspace{-0.3cm} \chi(\tau) \kappa (\tau)_t \  
	H [u_t ,\phi_{n-t}] \hat \tau \\
	& -  \sum_{k=2}^{n-1} \sum_{\sigma \in \s_{k,n-k}} \chi(\sigma) 
	{ (-1)^{k(n-k)}} \
	H(u_{n-k+1} + v_{n-k+1})  ( l_k \otimes \Id^{\otimes n-k}) \hat \sigma \\
	&- \sum_{k=3}^{n}  \sum_{i=2}^{k-1}  \sum_{\mu \in \s_{i,k-i,k,n-k}} \hspace{-0.5cm} \chi({\mu}) 
	(-1)^{{ n(k-1)+i(k-i)}} 
	H \phi_{n-k+1} (l_{k-i+1} \otimes \Id^{\otimes n-k} )   (l_i \otimes \Id^{\otimes n-i}) \hat \mu 	\\
	=&  -  u_n  - v_n  + (-1)^{n+1} \sum_{\mu \in \s_{1,n-1}} \phi_{n} ( \delta \otimes \Id^{\otimes n-1}) \hat \mu
	\end{align*}
	where the last equality follows from  Lemmas  \ref{lem:phi-u} and  \ref{lem: u,phi,l}.
\end{proof}

\section{A contraction of  $C(A)$ and $B(A)$} \label{Sec 3}

In this section we will show that there exists a contraction involving the Hochschild complex and Bardzell's complex.

Set $X=C_*(A)$ the Hochschild resolution, and $Y$ any projective resolution of the $A$-bimodule $A$.  Comparison morphisms 
\[ \begin{tikzpicture}
\node (A) at (0,0) {$Y$};
\node (B) at (3,0) {$C_*(A)$};
\draw[->] (A.20) -- node[above] {$F_*$} (B.170);
\draw[<-] (A.340) --  node[below] {$G_*$} (B.190);
\end{tikzpicture} \]
between these two projective resolutions are morphisms of complexes lifting the identity map on $A$.  It is clear that these morphisms exist, and that $G_*F_*, F_*G_*$ are homotopic to the identity maps $\Id_Y$ and $\Id_{C_*(A)}$ respectively, see for example \cite{W}. 

The next proposition will show how to construct recursively homotopy maps in terms of the homotopy $s_n$ defined in $\eqref{s_n}$.

\begin{lemma} \label{H} With the above notation, assume that $F_0G_0 = \Id$.
	An homotopy map between $\Id$ and $F_*G_*$ is given by the $A$-$A$-maps $$H_n: C_n(A) \to C_{n+1}(A)$$ defined recursively by $H_0=0$ and  $H_n= \Id \otimes
	F_n G_n s_{n-1}   - \Id \otimes  H_{n-1} d_n s_{n-1}$, that is, 
	$$H_n(1 \otimes x) : = (s_n F_n G_n  - s_{n} H_{n-1} d_n) (1 \otimes x), \quad \forall \, n\geq 1.$$
\end{lemma}

\begin{proof} Since $F_0G_0= \Id$, it is clear that $H_0=0$ satisfies $d_1 H_0 = F_0 G_0 - \Id$. Assume by induction that $d_{m+1} H_m + H_{m-1} d_m = F_m G_m - \Id$ holds for any $m < n$.  Then 
	$$ d_{n} H_{n-1} d_n = F_{n-1} G_{n-1} d_n- d_n = d_n F_n G_n - d_n.$$
	Using that 
	$s_{n-1} d_n + d_{n+1} s_n = \Id$ and that 
	$ s_{n-1} d_n (1 \otimes x)  = 1 \otimes x$,
	we have that
	\begin{align*}
	(d_{n+1} s_n F_n G_n)  (1 \otimes x)  & = (\Id - s_{n-1} d_n) F_n G_n  (1 \otimes x) \\
	& =  (F_nG_n - s_{n-1} d_n - s_{n-1} d_n H_{n-1} d_n)  (1 \otimes x) \\
	& = (F_nG_n - \Id + d_{n+1} s_{n} H_{n-1} d_n - H_{n-1} d_n)  (1 \otimes x)
	\end{align*}
	and hence the $A$-$A$-map $H_n$ defined by $$H_n(1 \otimes x) : = (s_n F_n G_n  - s_{n} H_{n-1} d_n) (1 \otimes x)$$
	satisfies the equation $d_{n+1} H_n + H_{n-1} d_n  = F_n G_n - \Id.$
\end{proof}

Let $A=\K Q/I $ be a monomial algebra.  It is well known that in this case $C_*(A)$ can be replaced by the exact complex
\[\overline C_*(A) = (A \otimes_E \rad A^{\otimes_E n} \otimes_E A, d_n)_{n \geq 0}\]
where $E= \K Q_0$, $A = E \oplus \rad A$ as $E$-bimodules, and $d_n$ is defined as in \eqref{d_n}, see \cite{C}.  Moreover,  the map $s_n$ defined now by 
\[ s_n (a \otimes x) = 1 \otimes a_r \otimes x\]
is an homotopy contraction, where $a= a_E+ a_r$.  Hence $s_n(1 \otimes x)=0$ and $s_{n-1} s_n =0$.  Lemma \ref{H} holds also for $\overline C_*(A)$
and in this case 
\begin{align}\label{radical}
H_n(1 \otimes x) = 1 \otimes F_n G_n  (1 \otimes x)  - 1 \otimes H_{n-1}x = ( \Id \otimes
F_n G_n  s_{n-1}   - \Id \otimes  H_{n-1}) (1 \otimes x).
\end{align}

Let $Y=B_*(A)$ be Bardzell's resolution for the monomial algebra $A=\K Q/I$. Even though the existence of comparison morphisms is clear, an explicit construction of these morphisms is not  always easy.  In \cite{RR1}, comparison morphisms between the projective resolutions $\overline C_*(A)$ and $B_*(A)$ have been explicitly described for any monomial algebra $A$. This description will allow us to prove that, when applying the functor $\Hom_{A-A}(-,A)$, we get 
\[ \begin{tikzpicture}
\node (A) at (0,0) {$B^*(A)$};
\node (B) at (3,0) {$\overline C^*(A)$};
\draw[->] (A.10) -- node[above] {$G^*$} (B.170);
\draw[<-] (A.350) --  node[below] {$F^*$} (B.190);
\path[->] (B) edge [out=0-18, in=0+18, loop]
node[right] {$H^*$.} (B);
\end{tikzpicture} \]
The equality $F^*G^* = \Id$ has been proved in \cite{RR1}.  In the forthcoming lemmas we will prove that all the vanishing conditions between $H^*, F^*$ and $G^*$ are satisfied, see \cite[Section 3]{RR1} for the definition of $F^*$ and $G^*$.  Since $F^n(f)=f F_{n}, G^n(f)=f G_{n}$ and $(-1)^{n-1} H^n (f) = f H_{n-1}$, it is enough to prove all the equalities for $F_*, G_*$ and $H_*$.

\begin{remark} \label{cuentas} For any $\alpha_1, \cdots, \alpha_n, \beta \in Q_1$ and for any $n \geq 1$, direct computations show that
	\begin{enumerate}
		\item 
		$H_1 (1 \otimes  \alpha_1 \cdots \alpha_n \otimes 1) = \sum_{i=2}^n 1 \otimes \alpha_1 \cdots \alpha_{i-1} \otimes \alpha_i \otimes \alpha_{i+1} \cdots \alpha_n$;
		\item
		$H_2 (1 \otimes \beta \otimes \alpha_1 \cdots \alpha_n \otimes 1) = \sum_{j=2}^n 1 \otimes \beta \otimes \alpha_1\dots \alpha_{j-1}\otimes \alpha_j \otimes \alpha_{j+1}\dots  \alpha_n$;
		\item 
		$H_n (1 \otimes \alpha_1 \otimes \cdots \otimes \alpha_n \otimes 1)=0$.
	\end{enumerate}
\end{remark}

\begin{lemma} For any $n \geq 0$ we have that 
	$$G_{n+1} H_n = \sum_{j=2}^{n}  (-1)^{n-j} G_{n+1} (\Id^{\otimes n-j+1} \otimes F_{j} G_{j} s_{j-1}).$$
\end{lemma}

\begin{proof}
	Let $X_m$ be the $A^e$-submodule of $ \overline C^m(A)=A \otimes_E \rad^m A \otimes_E A$ generated by the set 
	$$\{  1 \otimes v_1 \otimes \cdots \otimes v_m \otimes 1: v_{m-1}v_m \not \in I \}.$$  By definition, $G_m (1 \otimes v_1 \otimes \cdots \otimes v_m \otimes 1)=0$ if $v_{m-1}v_m \not \in I$, hence $G_m (X_m)=0$.
	We will prove, by induction, that 
	$$\Im (H_n - \sum_{j=2}^{n}  (-1)^{n-j} \Id^{\otimes n-j+1} \otimes F_{j} G_{j}s_{j-1}) \subset X_{n+1}.$$
	By Remark \ref{cuentas} we have $\Im H_1 \subset X_2$. Moreover,  
	$\Im (H_2 - \Id \otimes F_2G_2 s_1) = \Im \Id \otimes  H_{1} \subset X_3$.
	The lemma follows by the inductive hypothesis and by definition of  $H_n$ in $\overline C^n(A)$, see \eqref{radical}, since
	\begin{align*}
	H_n & -  \sum_{j=2}^{n}  (-1)^{n-j} \Id^{\otimes n-j+1} \otimes F_{j} G_{j}s_{j-1}  \\
	& = H_n - \Id \otimes F_nG_n s_{n-1} -   \sum_{j=2}^{n-1}  (-1)^{n-j} \Id^{\otimes n-j+1} \otimes F_{j} G_{j}s_{j-1} \\
	&= - \Id \otimes H_{n-1}  -   \sum_{j=2}^{n-1}  (-1)^{n-j} \Id^{\otimes n-j+1} \otimes F_{j} G_{j}s_{j-1} \\
	& = - \Id \otimes ( H_{n-1}  -   \sum_{j=2}^{n-1}  (-1)^{n-1-j} \Id^{\otimes n-j} \otimes F_{j} G_{j}s_{j-1}).
	\end{align*}
\end{proof}

\begin{lemma} If $j$ is odd, then $$G_{n+1} \circ (\Id^{\otimes n-j+1} \otimes  F_{j} G_{j}  s_{j-1}) = G_{n+1} \circ (\Id^{\otimes n-j+2} \otimes  F_{j-1} G_{j-1}  s_{j-2}).$$
\end{lemma}

\begin{proof}      Let $1 \otimes   v_{1} \otimes \cdots \otimes v_n \otimes 1 \in A \otimes \rad A^n \otimes A$.   Using the notation in \cite[Section 3]{RR1}, by definition we have that
	$$F_{j} G_{j} (1 \otimes v_{n-j+1} \otimes \cdots \otimes v_n \otimes 1)= \sum_i L_i F_j (1 \otimes w_i \otimes 1) R_i$$
	for all $w_i\in AP_j$ dividing  $v_{n-j+1}  \cdots  v_n$ with $s(w_i) < t(v_{n-j+1})$. 
	Since $j$ is odd we know that $|\operatorname{Sub}(w_i)|=2$; let $\psi_i\in \operatorname{Sub}(w_i)$ such that $w_i=L(\psi_i)\psi_i$. Then
	$$F_{j} G_{j} (1 \otimes v_{n-j+1} \otimes \cdots \otimes v_n \otimes 1)=  \sum L_i\otimes L(\psi_i)F_{j-1}(1\otimes \psi_i \otimes 1)R_i.$$
	If  $s(\psi_i) < t(v_{n-j+1})$, then $L_iL(\psi_i)\neq 0$ and 
	$$ G_{n+1}(1 \otimes v_1 \otimes \cdots \otimes v_{n-j} \otimes L_i\otimes  L(\psi_i)F_{j-1}(1\otimes \psi_i\otimes 1)R_i) =0.$$ 
	On the other hand, if $w_i$ is such that $s(\psi_i) \geq t(v_{n-j+1})$, then $\psi_i $ belongs to $AP_{j-1}$  and $s(\psi_i)$ is minimal with respect to all divisors of 
	$v_{n-j+2}  \cdots  v_n$ in  $AP_{j-1}$.  Hence there is a unique $w_i$ like this, 
	$$G_{j-1} (1 \otimes v_{n-j+2} \otimes \cdots \otimes v_n \otimes 1)= L'_i \otimes \psi_i\otimes R_i$$
	and the lemma follows.
\end{proof}

\begin{lemma}\label{lem:anula1} For any $n \geq 1$ we have that
	$G_{n+1} H_n =0$.
\end{lemma}

\begin{proof} 
	The proof follows by the two previous lemmas, and the fact that 
	$$ G_{n+1} (\Id \otimes F_nG_n s_{n-1} )=0$$
	if $n$ is even, since
	\begin{align*}
	G_{n+1} (1 \otimes F_nG_n (1 \otimes v_1 \otimes \cdots \otimes v_n \otimes 1) )&= 
	G_{n+1} (1 \otimes a F_n (1 \otimes w \otimes 1) b ) \\
	& = \sum_i  G_{n+1} (1 \otimes a \otimes u_1^i \otimes \cdots \otimes u_n^i \otimes 1) u_{n+1}^i b
	\end{align*}
	where $w \in AP_n$ is such that $s(w)$ is minimal.  The non vanishing of $$G_{n+1} (1 \otimes a \otimes u_1^i \otimes \cdots \otimes u_n^i \otimes 1)$$ would imply the existence of $z=(p_1, \cdots, p_n) \in AP_{n+1}$ starting in some vertex inside the path $a$ (not the last one). Then $(p_1, \cdots, p_{n-1})$ would contradict the minimality of $s(w)$. 
\end{proof}

\begin{lemma}\label{lem:anula2} For $n \geq 1$ we have 
	$H_nF_n =0$ and $H_{n+1} H_n=0.$
	Moreover, $(\Id \otimes H_n)F_{n+1}=0$ and $(\Id \otimes H_n)H_n=0$.
\end{lemma}

\begin{proof} The proof will be done by induction.  All the equalities hold for $n=1$.
	By definition $F_n (1 \otimes w \otimes 1) = 1 \otimes x$ for some $x$.  Then, using that $G_nF_n = \Id$ and the inductive hypothesis on the first equality, we get that
	\begin{align*} 
	H_n F_n (1 \otimes w \otimes 1) & = s_n F_n G_n F_n (1 \otimes w \otimes 1) - s_n H_{n-1} d_n F_n (1 \otimes w \otimes 1) \\
	& = s_n F_n (1 \otimes w \otimes 1) - s_n H_{n-1}  F_{n-1} d_{n-1}(1 \otimes w \otimes 1) =0.
	\end{align*} 
	Since $H_n (1 \otimes x) = 1 \otimes y$ for some $y$, 
	\begin{align*} 
	H_{n+1} H_n (1 \otimes x)  & = s_{n+1} F_{n+1} G_{n+1} H_n (1 \otimes x) - s_{n+1} H_n d_{n+1} H_n (1 \otimes x) \\
	& = - s_{n+1} H_n   (- H_{n-1} d_n + F_n G_n - \Id)  (1 \otimes x) =0.
	\end{align*} 
	Finally,  
	\begin{align*}
	(\Id \otimes H_n)F_{n+1} (1 \otimes w \otimes 1) & = (\Id \otimes H_n) (L(\psi) F_n (1 \otimes \psi \otimes 1) )= L(\psi) H_nF_n (1 \otimes \psi \otimes 1)=0; \\
	(\Id \otimes H_n)H_n (1 \otimes x) &= (\Id \otimes H_n)(1 \otimes F_nG_n(1 \otimes x) - 1\otimes H_{n-1}(x))=0.
	\end{align*}
\end{proof}

\begin{remark}\label{bardzell infinity}
	By the previous lemmas, the Hochschild complex $\overline C^*(A)$ and Bardzell's complex $B^*(A)$ satisfy the hypothesis of   \eqref{homotopy} and \eqref{ceros} used in Subsection \ref{subsec:Linf}. 
\end{remark}

\section{$L_\infty$-structure on Bardzell's complex}\label{sec:Barzdell}

All the algebras in this section are quotients of path algebras $\K Q/I$. The ideal $I$ is generated by a set $\R$ of paths that are minimal with respect to inclusion of paths, and $E=\K Q_0$.

We start this section with one of our main results, which can be deduced immediately from Remark \ref{bardzell infinity} and Theorem \ref{infinito}. Recall that a weak $L_\infty$-morphism $\phi=(\phi_n)$ is a weak equivalence if $\phi_1$ is a quasi-isomorphism.

\begin{theorem} \label{infinito-bardzell} For any monomial algebra $A$, let $B= B^*(A)[1]$ and $C=\overline C^*(A)[1]$ be Bardzell's and Hochshild complex, respectively.  The maps  $l_n:\otimes^n B\to B$, $n\in \N$, defined in Subsection \ref{subsec:Linf},  give $B$  a $L_\infty$-structure, and the quasi-isomorphism $G$ extends to a weak $L_\infty$-equivalence $\phi: B \to C$.
\end{theorem}

We present two examples that show that the behaviour of the  $L_\infty$-structure of $B^*(A)[1]$ may be complicated, and not nilpotent in general.

Recall from  \cite[page 17] {RR}  that a basis of  $B^n(A) = \Hom_{E-E}(\K AP_n, A)$ is given by the set of $\K$-linear maps $(w \vert \vert  \gamma)$ where $w  \in AP_n$, $\gamma$ is a path in $Q$, $\gamma  \not \in I$, $w$ is parallel to $\gamma$,  and
\[
(w \vert \vert \gamma)(\rho) = \begin{cases}
\overline \gamma, \qquad &\mbox{if $\rho=w$}; \\
0, & \mbox{otherwise.}
\end{cases}
\]

\begin{example} 
	Let $A=\K Q/I$ be the quadratic algebra given by      
	
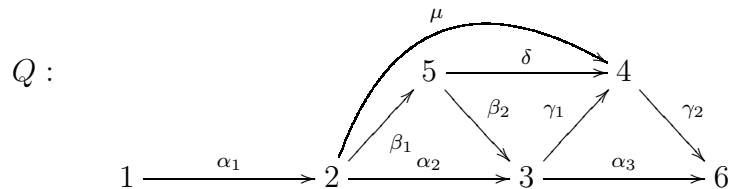
\[ \xymatrix{ Q: \hspace{2.5cm}
		& & 5 \ar[rd]^{\beta_2}  \ar[rr]^{\delta}  & & 4  \ar[rd]^{\gamma_2} \\
		1  \ar[r]^{\alpha_1} &    \ar@/^3.2pc/[rrru]^{\mu}   \ar[ru]_{\beta_1} 2 \ar[rr]^{\alpha_2} &   &  \ar[ru]^{\gamma_1}  3  \ar[rr]^{\alpha_3} & & 6
	} \] 
	and $I=< \alpha_1 \alpha_2, \alpha_2 \alpha_3, \beta_2 \alpha_3, \beta_2 \gamma_1, \beta_1 \delta>$.  Let
	$$f= (\alpha_1 \alpha_2 || \alpha_1 \beta_1 \beta_2) + (\beta_1 \delta || \alpha_2 \gamma_1 + \mu ) + (\beta_2 \gamma_1 || \delta) + ( \beta_2 \alpha_3 || \delta \gamma_2).$$
	In this case
	\[l_n (f^{\otimes n}) = \begin{cases} 
	(-1)^{\frac{(q-1)(q-2)}{2}}n!  \  ( (\alpha_1 \alpha_2, \alpha_2 \alpha_3) ||\alpha_1 \mu \gamma_2) & \mbox{if $n=3q$,} \\
	0  & \mbox{otherwise.} \
	\end{cases}\]
	
\end{example}

\begin{example}
	Let $A=\K Q/I$ be the quadratic algebra given by 
	\[ \xymatrix{ 
		Q: \qquad & 1   \ar[r]_{\alpha_1} &  \ar@/^0.5pc/[r]^\beta 2 \ar@/_0.5pc/[r]_{\alpha_2}    &  \ar@/^0.5pc/[r]^\gamma  3  \ar@/_0.5pc/[r]_{\alpha_3} & 4
	} \] 
	and $I=< \alpha_1 \alpha_2, \alpha_2 \alpha_3, \beta \gamma>$. Then $B^*(A) = B^0(A) \oplus B^1(A) \oplus B^2(A)$.
	Let 
	$$f=(\alpha_1 \alpha_2 || \alpha_1 \beta) + (\alpha_2 \alpha_3 || \alpha_2 \gamma) + (\beta \gamma || \alpha_2 \gamma + \beta \alpha_3) \in B^2(A).$$
	In this case
	\[l_n (f^{\otimes n}) = \begin{cases} 
	(-1)^{\frac{n-1}{2}}n!  \  ( (\alpha_1 \alpha_2, \alpha_2 \alpha_3) || \alpha_1 \beta \alpha_3), & \mbox{if $n$ is odd;} \\
	0,  & \mbox{otherwise.} \\
	\end{cases}\]
\end{example}

Recall that if $L$ is a $L_\infty$-algebra, the set $\MC(L)$ of Maurer-Cartan elements consists of all $f \in L^1$ satisfying the generalized Maurer-Cartan equation
$$l_1 (f) - \sum_{n \geq 2} (-1)^{\frac{(n+1)n}{2}} \frac{1}{n!} l_n(f\otimes \cdots \otimes f) =0.$$
Given an algebra $A$, it is well known that the formal deformations of $A$ over $\K [[t]]$ are in one-to-one correspondence with equivalence classes of Maurer-Cartan elements in $\MC(\overline C^*(A)[1] \otimes ((t)))$, see for instance \cite[\S 5]{DMZ}.
In the particular case of monomial algebras, Theorem \ref{infinito-bardzell}  and  \cite[Theorem 7.8]{DMZ}  implies that 
$$\MC (\overline C^*(A)[1] \otimes ((t))) \simeq \MC (B^*(A)[1] \otimes ((t)))$$
where  $B^*(A)[1]= (B^{n+1}(A), -(-1)^{n+1} \delta^{n+1}, l_n)$ and  $l_1 = - \delta^2$.  In this case $f = \sum_{i \geq 1} f_i t^i$ with $f_i \in B^1(A)[1]$ satisfies the 
generalized Maurer-Cartan equation if 
$$-  \delta^2(f_i) - \sum_{n \geq 2}  \sum_{j_1 + \cdots + j_n=i} (-1)^{\frac{(n+1)n}{2}} \frac{1}{n!} l_n (f_{j_1} \otimes \cdots \otimes f_{j_n}) =0.$$

In order to get some general results concerning the $L_\infty$-structure of $B^*(A)[1]$ we continue with some assertions that, using the inductive definition of $l_n$,
should allow us to find Maurer-Cartan elements in some particular cases.

\begin{remark}\label{lem: trunc} Let $f_1, f_2\in B^{1}(A)[1] =  \Hom_{E-E}( \K AP_{2}, A)$. Then 
	$$l_2(f_1\otimes f_2) = F^{3} [G^{2}(f_1), G^{2}(f_2)] =  \sum_{i=1}^2 G^2 (f_i) (G^2 (f_{2-i}) \otimes \Id - \Id \otimes G^2 (f_{2-i}) )F_3. $$
\end{remark}
Observe that $\vert \phi_n (f_1 \otimes \cdots \otimes f_n)\vert =1$ when $\vert f_i \vert =1$ for all $i$.  Hence \eqref{circ} is the formula we use in the following results when computing $[\phi_t, \phi_{n-t}]$ in the shifted complex $\overline C^*(A)[1]$.

\begin{lemma}\label{v_n anula} Let $f_1, f_2, \cdots, f_n \in B^{1}(A)[1] = \Hom_{E-E}( \K AP_{2}, A)$, $\alpha, \beta \in Q_1$ and $v \in \rad A$. Assume that 
	$\alpha v$ and $v \beta$ are non zero or correspond to  paths in $\R$.  Then, for any $n \geq 2$,
	$$\phi_n (f_1 \otimes \cdots \otimes f_n) (1 \otimes \alpha \otimes v \otimes 1)=0 = \phi_n (f_1 \otimes \cdots \otimes f_n) (1 \otimes v \otimes \beta \otimes 1).$$
\end{lemma}

\begin{proof} It is clear that $ \phi_n (f_1 \otimes \cdots \otimes f_t) (1 \otimes v \otimes \beta \otimes 1)=0$ since $H_2( 1 \otimes v \otimes \beta \otimes 1)=0$.  For the other equation, we will proceed by induction on $n$. For $n=2$,
	if $v= \beta_1 \cdots \beta_s$, 
	$$H_2 (1 \otimes \alpha \otimes v \otimes 1)  = \sum_{j=2}^s 1 \otimes \alpha \otimes \beta_1\dots \beta_{j-1}\otimes \beta_j \otimes \beta_{j+1}\dots \beta_s.$$
	Thus, since $G_2(1 \otimes \alpha \otimes   \beta_1\dots \beta_{j-1} \otimes 1)=0=G_2(1 \otimes  \beta_1\dots \beta_{j-1} \otimes \beta_{j} \otimes 1)$, we get
	$$\phi_2(f_1,f_2)  (1 \otimes \alpha \otimes v  \otimes 1)= [G^2f_1,G^2f_2] H_2(1 \otimes \alpha \otimes v  \otimes 1)=0.$$ 
	For the inductive step,  we have that 
	$$\phi_n= H^2 v_n = H^2  \sum_{t=1}^{n-1} \sum_{\tau \in \s^{-}_{t,n-t}}  \chi(\tau) \kappa(\tau)_t \ [\phi_t, \phi_{n-t}] \hat \tau.$$
	Hence, for any $n-t \geq 2$, the inductive hypothesis and Lemma \ref{lem:anula2} imply that
	\begin{align*}
	\phi_t (\phi_{n-t} \otimes \Id & - \Id \otimes \phi_{n-t}) (f_{\tau(1)} \otimes \cdots \otimes f_{\tau(n)})H_2 (1 \otimes \alpha \otimes v  \otimes 1)   \\
	=& \phi_t (\phi_{n-t} \otimes \Id) (f_{\tau(1)} \otimes \cdots \otimes f_{\tau(n)}) H_2 (1 \otimes \alpha \otimes v  \otimes 1) \\
	& -\phi_t (\Id \otimes v_{n-t}) (f_{\tau(1)} \otimes \cdots \otimes f_{\tau(n)})  (\Id \otimes H_2) H_2  (1 \otimes \alpha \otimes v \otimes 1) =0.
	\end{align*}
	Finally, 
	$$
	\phi_{n-1}(\phi_1 \otimes \Id   - \Id \otimes \phi_1) (f_{\tau(1)} \otimes \cdots \otimes f_{\tau(n)})H_2 (1 \otimes \alpha \otimes v  \otimes 1)=0$$
	since $\phi_1(f)=f G_2$ and  $G_2(1 \otimes \alpha \otimes  \beta_1\dots \beta_{j-1} \otimes 1)=0=G_2(1 \otimes  \beta_1\dots \beta_{j-1} \otimes \beta_j \otimes 1)$.
	
\end{proof}

\begin{proposition} \label{veremos} Let $f_1,  \cdots, f_n \in B^{1}(A)[1] = \Hom_{E-E}( \K AP_{2}, A)$, $\alpha, \beta \in Q_1$, $v \in \rad A$.  Then, for any $n \geq 3$, 
	$$v_n (f_1 \otimes \cdots \otimes f_n)(1 \otimes \alpha \otimes v \otimes \beta \otimes 1) =\begin{cases} 0, & \mbox{if $\alpha v, v \beta \not =0$}; \\
	A,  & \mbox{if $\alpha v \in \R, v \beta \not =0$}; \\
	-B,  & \mbox{if $\alpha v \not =0, v \beta \in \R$}; \\
	A-B, & \mbox{if $\alpha v, v \beta \in \R$};
	\end{cases}$$
	where 
	\begin{align*}
	A & =   (-1)^n  \sum_{i=1}^n \phi_{n-1} (f_1 \otimes \cdots \otimes  \hat f_i \otimes \cdots \otimes f_n ) (1 \otimes f_i(\alpha v)  \otimes  \beta \otimes 1), \mbox{ and} \\
	B  &=  (-1)^{n} \sum_{i=1}^n \phi_{n-1} (f_1 \otimes \cdots \otimes  \hat f_i \otimes \cdots \otimes f_n ) (1 \otimes \alpha  \otimes f_i(v \beta) \otimes 1). 
	\end{align*}
\end{proposition}

\begin{proof} The result follows from Lemma \ref{v_n anula} since, for any $n-t \geq 2$, 
	\begin{align*}\phi_t (\phi_{n-t} \otimes \Id  - \Id \otimes \phi_{n-t})(f_{\tau(1)} \otimes \cdots \otimes f_{\tau(n)})(1 \otimes \alpha \otimes v \otimes \beta \otimes 1)  =0,\\
	\phi_1(f) = G^2(f) \quad \mbox{ and } \quad G^2(f) (1 \otimes \alpha \otimes v \otimes 1)= \begin{cases} 0 & \mbox{if $\alpha v \not =0$}\\
	f(\alpha v)& \mbox{if $\alpha v \in \R$}.
	\end{cases}
	\end{align*}
	Similarly for $v \beta$.
\end{proof}

\subsection{Radical square zero algebras}

In this subsection we assume that $A$ is a radical square zero algebra, that is, $A=\K Q/J^2$, where $J$ is the two sided ideal of $\K Q$ generated by the arrows.

\begin{theorem}
	If  $A=\K Q/J^2$  then the shifted Bardzell's complex $B^*(A)[1]$ is a dg-Lie algebra and 
	$G^*: \overline C^*(A) \to B^*(A)$ is a quasi-isomorphism of dg-Lie algebras. 
	In particular, there is a bijection $$\MC(\overline C^*(A) \otimes (t) ) \simeq \MC(B^*(A) \otimes (t))$$ and  $f = \sum_{i\geq 1} f_i t^i  \in \Hom_{E-E}(\K J^2, A) \otimes (t) $ satisfies the Maurer-Cartan equation if and only if, for any $(\alpha_1 \alpha_2, \alpha_2 \alpha_3) \in AP_3$,
	$$f_i(\alpha_1 \alpha_2) \alpha_3 - \alpha_1 f_i (\alpha_2 \alpha_3) +   \sum_{j+k=i} f_j G_2(1 \otimes f_k(\alpha_1 \alpha_2) \otimes \alpha_3 \otimes 1-1 \otimes \alpha_1 \otimes f_k(\alpha_2 \alpha_3) \otimes 1)=0.$$
\end{theorem}

\begin{proof} From Remark \ref{cuentas} we have that $H_n=0$ for all $n$. Then 
	$$\phi_t(f_1 \otimes \cdots \otimes f_t) =  H^{m+1} v_t(f_1 \otimes \cdots \otimes f_t) = (-1)^m v_t(f_1 \otimes \cdots \otimes f_t) H_m =
	0$$
	for any $t \geq 2$. 
	Moreover, $v_n$ is defined in terms of $[\phi_t, \phi_{n-t}]$ with $1 \leq t <n$, and hence,  for any $n >2$, $l_n = F^s v_n =0$.  Then $B^*(A)[1]$ is a dg-Lie algebra and $\phi_1=G$ is a quasi-isomorphism of dg-Lie algebras.   In this case, the Maurer-Cartan equation reduces to 
	$-\delta^2 (f) + \frac 12 l_2 (f \otimes f) =0$ and, 
	for any $(\alpha_1 \alpha_2, \alpha_2 \alpha_3) \in AP_3$,
	\begin{align*}
	-\delta^2(f_i) (\alpha_1 \alpha_2, \alpha_2 \alpha_3) & = f_i(\alpha_1 \alpha_2) \alpha_3 - \alpha_1 f_i (\alpha_2 \alpha_3), \\
	G^2 (f_i) (G^2 (f_j) \otimes \Id) F_3 (1 \otimes  (\alpha_1 \alpha_2, \alpha_2 \alpha_3) \otimes 1)
	&= f_i G_2 (1 \otimes f_j(\alpha_1 \alpha_2) \otimes \alpha_3 \otimes 1), \mbox{ and} \\
	G^2 (f_i) (\Id \otimes G^2 (f_j)) F_3 (1 \otimes  (\alpha_1 \alpha_2, \alpha_2 \alpha_3) \otimes 1)
	&= f_i G_2 (1 \otimes \alpha_1 \otimes f_j( \alpha_2 \alpha_3) \otimes 1).
	\end{align*}
	
\end{proof}

The previous theorem allows us to find a formula for the Gerstenhaber bracket $[u,v]$ for any  $u,v \in \HH^2(A)$ when $\rad^2 A=0$ by using Bardzell's complex since $[u,v] = G^3 l_2 (F^2(u) \otimes F^2(v))$.  The importance of this result relies in the fact that Bardzell's complex has shown to be more efficient than Hochschild complex when making concrete computations.

It is clear that it suffices to compute $[f, g]$ for any $f, g$ in the described basis of $B^*(A)$. In particular, the compositions corresponding to 
the first and the second summand in \eqref{circ} for elements in $C^2(A)$ behave as follows in $B^2(A)$:
\begin{align*}
(\beta \alpha_3 \vert \vert \gamma) \circ_0 (\alpha_1 \alpha_2 \vert \vert \beta) = (\alpha_1 \alpha_2 \alpha_3 \vert \vert  \gamma), \\
(\alpha_1 \beta \vert \vert \gamma) \circ_1 (\alpha_2 \alpha_3 \vert \vert \beta) = (\alpha_1 \alpha_2 \alpha_3 \vert \vert \gamma),
\end{align*}
with $\beta \in Q_1$, $\gamma \in Q_0 \cup Q_1$, and all the other cases vanish.

One can observe that the bracket may be non-zero only if the quiver contains a subquiver as follows
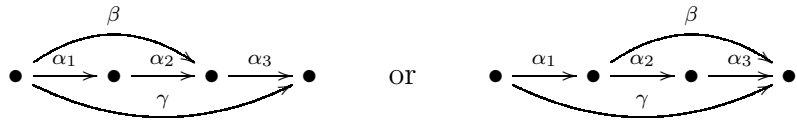
\[
\xymatrix{ 
	\bullet  \ar@/^1.3pc/[rr]^\beta  \ar@/_1.3pc/[rrr]^\gamma \ar[r]^{\alpha_1} & \bullet \ar[r]^{\alpha_2}  & \bullet  \ar[r]^{\alpha_3} & \bullet 
} \qquad \mbox{or} \qquad
\xymatrix{
	\bullet   \ar@/_1.3pc/[rrr]^\gamma \ar[r]^{\alpha_1} & \bullet \ar[r]^{\alpha_2} \ar@/^1.3pc/[rr]^\beta   & \bullet  \ar[r]^{\alpha_3} & \bullet 
}\]
where the drawn vertices are not necessarily different, and hence the subquiver may contain loops and cycles. 
We can conclude that if the quiver does not contain subquivers as above, Maurer-Cartan elements {of $B^*(A)[1]$} are in one to one correspondence with Hochschild $2$-cocycles. { Moreover, the Maurer-Cartan elements of $B^*(A)[1] \otimes ((t))$ are of the form $\sum_{i\geq 1} f_i t^i$ with $f_i$ Hochschild $2$-cocycles.}
When the quiver contains a subquiver as above, then we cannot deduce a relation between Maurer-Cartan elements and Hochschild $2$-cocycles as we note in the following examples.

\begin{example}
	Let $A=\K Q/ I$ be the radical square zero algebra with 
	
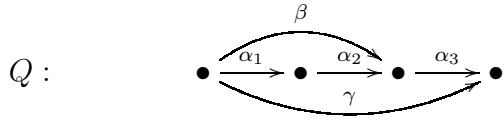
\[
	\xymatrix{ 
		Q: \qquad & \bullet  \ar@/^1.3pc/[rr]^\beta  \ar@/_1.3pc/[rrr]^\gamma \ar[r]^{\alpha_1} & \bullet \ar[r]^{\alpha_2}     & \bullet  \ar[r]^{\alpha_3} & \bullet
	} \] 
	Let $f=(\alpha_1 \alpha_2 \vert \vert \beta) +  (\beta \alpha_3 \vert \vert \gamma)$, then 
	\begin{align*}
	&\delta^2 (f )(\alpha_1\alpha_2,\alpha_2\alpha_3)= 0, &
	&l_2(f,f)(\alpha_1\alpha_2,\alpha_2\alpha_3)= 2\gamma.
	\end{align*}
	Thus, $f$ is a Hochschild 2-cocycle but it is not a Maurer-Cartan element {of $B^*(A)[1]$}.	
\end{example}

\begin{example}
	Let $A=\K Q/ I$ be the radical square zero algebra with  quiver
	
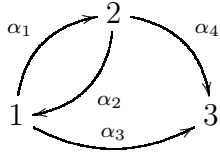
\[ \xymatrix{ 
		& 2 \ar@/^1pc/[ld]^{\alpha_2} \ar@/^1pc/[dr]^{\alpha_4} \\
		1 \ar@/^1pc/[ru]^{\alpha_1} \ar@/_1pc/[rr]^{\alpha_3}  & & 3
	} \] 
	Let $f_1= (\alpha_2\alpha_3 \vert \vert \alpha_4)+ (\alpha_1\alpha_4 \vert \vert \alpha_3)$ and $f_2=(\alpha_1\alpha_2 \vert \vert e_1) + (\alpha_2\alpha_1 \vert \vert e_2)$.
	Using the formulas above one can check that
	\begin{align*}
	&\delta^2( f_2) (\alpha_1\alpha_2,\alpha_2\alpha_3)=- \alpha_3,
	&
	&\delta^2( f_2)(\alpha_2\alpha_1,\alpha_1\alpha_4)=-\alpha_4, \\
	&l_2(f_1,f_1)(\alpha_1\alpha_2,\alpha_2\alpha_3)=- 2\alpha_3,
	& 
	&l_2(f_1,f_1) (\alpha_2\alpha_1,\alpha_1\alpha_4)= -2\alpha_4,
	\end{align*}
	and $0$ otherwise. In this case $f=f_1+f_2$ is a Maurer-Cartan element  of $B^*(A)[1]$ but it is not a Hochschild 2-cocycle.
	Moreover, $f=f_1t+f_2t^2$ is a Maurer-Cartan element of  $B^*(A)[1]\otimes ((t))$ since
	$$l_1(f)+\frac{1}{2}l_2(f,f)= -\delta^2(f_1)t+(-\delta^2(f_2)+\frac{1}{2}l_2(f_1,f_1))t^2 +l_2(f_1,f_2)t^3+\frac{1}{2}l_2(f_2,f_2)t^4=0.$$	
	
\end{example}

\subsection{Truncated quiver algebras}
In this subsection we  assume that $A$ is a truncated quiver algebra, that is, $A=\K Q/J^n$ for some $n\geq 2$, where $J$ is the two sided ideal of $\K Q$ generated by the arrows.

\begin{proposition}
	Let $A=\K Q/J^n$ for some $n\geq 2$ and let $f_1, \cdots, f_m \in B^1(A)[1]$, then $$l_m(f_1 \otimes \cdots \otimes f_m)=0, \quad \forall\, m\geq 3.$$
\end{proposition}
\begin{proof} Let $w\in AP_3$, that is, $w=(\alpha_1\dots \alpha_n, \alpha_2\dots \alpha_{n+1})$.
	Then 
	\begin{align*}l_m(f_1 \otimes \cdots \otimes f_m)(w) &= v_m(f_1 \otimes \cdots \otimes f_m)F_3(w), \\
	F_3(1\otimes w\otimes 1)&= \sum_{i=2}^n 1\otimes \alpha_1\otimes\alpha_2\dots\alpha_i\otimes \alpha_{i+1}\otimes \alpha_{i+2}\dots\alpha_{n+1}.
	\end{align*}
	From  Proposition \ref{veremos} we only have to prove that 
	$$v_m(f_1 \otimes \cdots \otimes f_m) (1\otimes \alpha_1\otimes\alpha_2\dots\alpha_n\otimes \alpha_{n+1}\otimes 1)=0,$$
	and this follows from Lemma \ref{v_n anula}.
\end{proof}

\begin{example} 
	Let $A=\K Q/ J^n$, $n\geq 2$ be the algebra with quiver
	
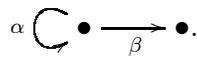
\[
	\xymatrix{ 
		\ar@(ul,dl)_{\alpha} \bullet \ar[r]_{\beta} & \bullet .
	}  \] 
	Let  $f_1= (\alpha^{n-1}\beta\vert \vert \alpha^{n-2} \beta)$ and $f_2=(\alpha^n \vert \vert \alpha^{n-2})$.
	One can check that, 
	\begin{align*}
	\delta^2( f_2 )(\alpha^n, \alpha^{n-1}\beta)= -\alpha^{n-2}\beta, 
	\qquad l_2(f_1,f_1)(\alpha^n, \alpha^{n-1}\beta)=     -2\alpha^{n-2}\beta,
	\end{align*}
	and $0$ otherwise. Then $f=f_1+f_2$ is a Maurer-Cartan element  of $B^*(A)[1]$ and $f=f_1t+f_2t^2$ is a Maurer-Cartan element of  $B^*(A)[1]\otimes ((t))$.

\end{example}

As the following example shows, we cannot expect to have dg-Lie algebras in general.

\begin{example}
	Let  $A=\K Q/J^3$ be the algebra whose quiver $Q$ is given by
	
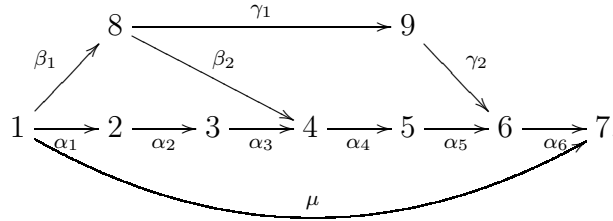
\[ \xymatrix{ 
		& 8 \ar[rrd]^{\beta_2}  \ar[rrr]^{\gamma_1} & & &  9  \ar[rd]^{\gamma_2} & & & \\
		1  \ar[r]_{\alpha_1} \ar[ru]^{\beta_1}   \ar@/_2.8pc/[rrrrrr]^{\mu}  &  2 \ar[r]_{\alpha_2}    &   3  \ar[r]_{\alpha_3}  & 4 \ar[r]_{\alpha_4} & 5 \ar[r]_{\alpha_5} &  6\ar[r]_{\alpha_6} &  7 .
	} \] 
	Let $f_1= ( \alpha_1\alpha_2\alpha_3 || \beta_1\beta_2), f_2= ( \beta_2\alpha_4\alpha_5 || \gamma_1\gamma_2) \in B^1(A)[1]$ and $f_3= ((\beta_1\gamma_1\gamma_2,\gamma_1\gamma_2\alpha_6) || \mu) \in B^2(A)[1]$.
	Then $$l_3(f_1,f_2,f_3)(\alpha_1\alpha_2\alpha_3,\alpha_2\alpha_3\alpha_4,\alpha_4\alpha_5\alpha_6)=-\mu.$$

\end{example}

\end{document}